\documentclass[11pt]{elsarticle}
\usepackage{amsmath,amssymb,amsthm}
\usepackage{url,doi}
\usepackage{hyperref}
\usepackage{graphicx}
\usepackage{algpseudocode}
\usepackage{algorithm}
\usepackage{caption}

\usepackage{fullpage}
\usepackage{enumerate}

\newtheorem{theorem}{Theorem}

\newtheorem{theorem*}{Theorem}

\journal{Journal of Computer and System Sciences}

\begin{document}

\begin{frontmatter}

\title{Maximizing the number of integer pairs summing to powers of $2$ via graph labeling and solving restricted systems of linear (in)equations}
\author{Max A. Alekseyev
}

\affiliation{organization={The George Washington University}, %Department and Organization
            addressline={},
            city={Washington},
            postcode={},
            state={DC},
            country={USA}
}

\begin{abstract}
We address the problem of finding sets of integers of a given size with a maximum number of pairs summing to powers of $2$.
By fixing particular pairs, this problem reduces to finding a labeling of the vertices of a given graph with pairwise distinct integers such that the endpoint labels for each edge sum up to a power of $2$.
We propose an efficient algorithm for this problem, 
which at its core relies on another algorithm that, given two sets of linear homogeneous polynomials with integer coefficients, computes all variable assignments to powers of $2$ that nullify polynomials from the first set but not from the second.
With the proposed algorithms, we determine the maximum size of graphs of order $n$ that admit such a labeling for all $n\leq 21$, and construct the maximum admissible graphs for $n\leq 20$.
We also identify the minimal forbidden subgraphs of order $\leq 11$, whose presence prevents the graphs from having such a labeling.
\end{abstract}

\begin{keyword}
graph labeling \sep extremal set \sep algorithm

\MSC 05C35 \sep \MSC 05C85 \sep \MSC 05D05 \sep \MSC 11Y50

\end{keyword}

\end{frontmatter}

\section{Introduction}

In April 2021, Dan Ullman and Stan Wagon published ``Problem of the Week 1321''~\cite{Wagon2021} (as a generalization of AMM Problem 12272~\cite{ShahAli2021}),
where they introduced a function $f(A)$ of a finite set $A$ of integers as the number of $2$-element subsets of $A$ that sum to a power of $2$.
They gave an example $f\big(\{-1, 3, 5\}\big) = 3$ and further defined a function
$g(n)$ as the maximum of $f(A)$ over all $n$-element sets $A$. The problem asked for a proof that $g(10)\geq 14$, which was quickly improved to $g(10)\geq 15$ by the readers, 
and Pratt further showed\footnote{Pratt's proof apparently is unpublished and only mentioned in~\cite[Solution]{Wagon2021}.} that, in fact, $g(10) = 15$.

Multiple researchers have noted that computing $g(n)$ has a natural interpretation as finding a maximum graph of order $n$, where the vertices are labeled with pairwise distinct integers,
and the sum of the endpoint labels for each edge is a power of $2$ (see Fig.~\ref{fig:msg16} for examples).
In March 2022, M. S. Smith proved\footnote{Smith emailed a proof to Neil Sloane, and it is quoted in the OEIS entry {\tt A352178}~\cite{OEIS}.} that such a graph cannot contain a cycle $C_4$, limiting the candidate graphs to well-studied \emph{squarefree graphs}~\cite{Clapham1989}. 
Smith's result made it easy to establish the values of $g(n)$ for all $n\leq 9$, and led to the creation of the sequence {\tt A352178} in the Online Encyclopedia of Integer Sequences (OEIS)~\cite{OEIS}.
It also provided a nontrivial upper bound for $g(n)$, namely the maximum size of squarefree graphs of order $n$ (given by sequence {\tt A006855} in the OEIS).

The problem received further attention in September 2022 after Neil Sloane presented it on the popular Numberphile Youtube channel~\cite{Haran2022}. 
This was followed by a few improvements, including the value\footnote{Apparently, the exact value of $g(10)$ was not present in the OEIS back then, and most people were unaware that Pratt had already established that $g(10)=15$.} $g(10)=15$ from Matthew Bolan~\cite{Bolan2022}, and independently the same value and $g(11)=17$ from Firas Melaih~\cite{Melaih2022}.
These results were obtained via a graph-theoretical treatment of the problem by manually analyzing a few candidate graphs.

At the same time, two methods were proposed to obtain lower bounds for $g(n)$, giving those for the first few hundreds of values of $n$, which are listed in the sequences {\tt A347301} and {\tt A357574} in the OEIS.

In the present paper, we propose an algorithm for testing the \emph{admissibility} of a given unlabeled graph, i.e., whether its vertices can be labeled with pairwise distinct integers such that
the sum of the endpoint labels for each edge is a power of $2$. 
At its core, this algorithm relies on another proposed algorithm that, given two sets of linear homogeneous polynomials with integer coefficients, computes all variable assignments to powers of $2$ that nullify polynomials from the first set but not from the second. The polynomials from the first and second sets are referred to as equations and inequations, respectively.

We use our algorithms to bound $g(n)$ from above, and together with the known lower bounds to establish the values of $g(n)$ for $n$ in the interval $[12,21]$.
Also, interpreting Smith's result as saying that the cycle $C_4$ is a \emph{minimal forbidden subgraph} (MFS), i.e., inadmissible graphs in which all proper subgraphs are admissible, 
we find all MFSs of orders up to $11$. Namely, we show that there are no MFSs of order $5$, $6$, $8$, or $9$, while there are $2$ MFSs of order $7$ (Fig.~\ref{fig:mfs7}), 
$15$ MFSs of order $10$, and $77$ MFSs of order $11$.
We also enumerate the \emph{maximum admissible graphs} (MAGs), i.e. admissible graphs of the maximum size, of orders up to $20$. 
We establish that there are $4$, $28$, $2$, $18$, $2$, $22$, and $2$ MAGs of order from $14$ to $20$, respectively.

\section{Algorithm for testing graph admissibility}

A given graph $G$ on $n$ vertices with $m$ edges is admissible if and only if the following matrix equation is soluble:
\begin{equation}\label{mateq}
M \cdot L = X,
\end{equation}
where 
\begin{itemize}
\item
$M$ is the $m\times n$ incidence matrix of $G$ with rows and columns indexed by the edges and vertices of $G$, that is, $M$ is a $\{0,1\}$-matrix with each row containing exactly two $1$'s;
\item
$L = (l_1,l_2,\dots,l_n)^T$ is a column vector of pairwise distinct integer vertex labels;
\item
$X = (x_1,x_2,\dots,x_m)^T$ is a column vector formed by powers of $2$ representing the sums of edges' endpoint labels.\footnote{The elements of $X$ are not required to be distinct.}
\end{itemize}
Both $L$ and $X$ in this equation are unknown and have to be determined.

We start by solving \eqref{mateq} for $L$ in terms of $X$. That is, we compute a (partial) solution of the form $l_i = p_i(x_1,\dots,x_m)$ for $i\in\{1,2,\dots,n\}$, where each $p_i$ is a linear polynomial with rational coefficients.
Such a solution exists if and only if $K_l\cdot X = 0$, where $K_l$ is an integer matrix with rows that form a basis of the left kernel of $M$. 
It will be seen that the performance of our algorithm is sensitive to the number of nonzero elements in $K_l$, and thus we assume that $K_l$ is sparse, which can be achieved (to some extent) with LLL reduction~\cite{Bremner2011}.
Let $E$ be the set of elements of $K_l\cdot X$, which are homogeneous linear polynomials with integer coefficients representing linear equations in $x_1,\dots,x_m$.

Let $K_r$ be a matrix whose columns form a basis of the right kernel of $M$. For a connected graph $G$, it is known that $K_r$ has size $n\times t$, where $t=1$ or $t=0$ depending on whether $G$ is bipartite~\cite{Nuffelen1976}.
We find it convenient to view an $n\times 0$ matrix as composed of $n$ empty rows (and thus having all rows equal).
Adding a linear combination of the columns of $K_r$ to a solution $L$ of equation \eqref{mateq} turns it into another solution $L'$, and furthermore all solutions with the same $X$ can be obtained this way.
The following theorem implies that for any set of pairwise distinct rows of $K_r$, we can find a linear combination of the columns of $K_r$ 
such that the corresponding elements of $L'$ are also pairwise distinct.

\begin{theorem}\label{th:distinct}
Let $v$ be an integer column vector of size $k\geq 0$, and $A$ be a $k\times s$ integer matrix with pairwise distinct rows. 
Then there exists an integer linear combination of the columns of $A$ such that adding it to $v$ results in a vector with pairwise distinct elements.
\end{theorem}

\begin{proof} If $s=0$, then with necessity we have $k=1$, and thus $v$ already has pairwise distinct elements. 

Let us prove the statement for $s=1$. In this case, $A$ represents a column vector with pairwise distinct elements. 
Let $t$ be the difference between the largest and smallest elements of $v$. It is easy to see that the vector $v + A\cdot (t+1)$ has pairwise distinct elements.

In the case of $s>1$, let $d$ be the difference between the largest and smallest elements of $A$.
Then the $k\times 1$ matrix $A' := A\cdot (1,(d+1),(d+1)^2,\dots,(d+1)^{s-1})^T$ has pairwise distinct elements, 
thus reducing the problem to the case $s=1$ considered above.
\end{proof}

Theorem~\ref{th:distinct} implies that we can restrict our attention only to the pairs of elements of $L$ corresponding to the equal rows in $K_r$.
For each pair of equal rows in $K_r$ with indices $i<j$, we compute $q(x_1,\dots,x_m) := (p_i(x_1,\dots,x_m)-p_j(x_1,\dots,x_m))c$, where $c$ is a positive integer factor making all coefficients of $q$ integer.
If $q$ is a zero polynomial, then condition $l_i\ne l_j$ is unattainable, and thus the graph $G$ is inadmissible.
On the other hand, if $q$ consists of just a single term with a nonzero coefficient, then the condition $l_i\ne l_j$ always holds, and we ignore such $q$.
In the remaining case, when $q$ contains two or more terms with nonzero coefficients, we add $q$ to the set $N$.
The resulting set $N$ consists of polynomials in $x_1,\dots,x_m$ that must take nonzero values on a solution to \eqref{mateq}.
We refer to such polynomials as \emph{inequations}.\footnote{We deliberately use the term \emph{inequation} to denote relationship $p\ne 0$ 
and to avoid confusion with inequalities traditionally denoting relationships $\geq$, $\leq$, $>$, or $<$.}

Our next goal is to solve the system of equations $E$ and inequations $N$ in powers of $2$, which we address in Section~\ref{sec:solvep2}.
For each solution $X=X_0$, we plug it into the system~\eqref{mateq} turning it into a standard matrix equation, which we then solve for $L$.
The resulting solution can be parametric with parameters representing the free variables in $X_0$ (if any) and the coefficients of a linear combination of the columns of $K_r$ when it is not empty. 
In either case, from such a solution Theorem~\ref{th:distinct} guarantees to produce a solution composed of pairwise distinct integers.

We outline the above description in Algorithm~\ref{alg:GraphSolve}. While this algorithm can produce all solutions (i.e., labelings of the graph vertices), when testing graph admissibility we stop it as soon as one solution is obtained.\footnote{Technically $\textsc{GraphSolve}$ is implemented as a SageMath/Python's generator, which produces solutions one after another on demand and which we can easily stop after obtaining just one solution (when it exists).}

\begin{algorithm}[!t]
\caption{An algorithm for solving system~\eqref{mateq} for a given graph $G$.}
\label{alg:GraphSolve}
\begin{algorithmic}[1]
\begin{footnotesize}
\Function{GraphSolve}{$G$}
\State Set $E:=\emptyset$ and $N:=\emptyset$
\State Let $m$ and $n$ be the size and order of $G$, respectively.
\State Construct the $m\times n$ incidence matrix $M$ of $G$ with rows and columns indexed by the edges and vertices of $G$
\State Compute a sparse basis $K_l$ of the left kernel of $M$. \Comment{We have $K_l\cdot M=0$.}
\State Let $X:=(x_1,\dots,x_m)^T$ be the column-vector formed by indeterminates.
\For {each row $r$ in $K_l$}
\State Add the polynomial $r\cdot X$ to $E$
\EndFor
\State Solve $ML=X$ for $L$ in terms of $X$, let $(p_1,\dots,p_n)$ be any particular solution.
\State Compute $K_r$ whose columns form a basis of the right kernel of $M$. \Comment{We have $M\cdot K_r=0$.}
\For {each $\{i,j\}\subset \{1,2,\dots,n\}$}
\If {$i$th and $j$th rows of $K_r$ are not equal}
\State \textbf{continue} to next subset $\{i,j\}$    \Comment{Per Theorem~\ref{th:distinct} we ignore such pair of indices.}
\EndIf
\State Set $q$ equal to a multiple of $p_i-p_j$ with integer coefficients
\If {$q=0$}
\State \Return {$\emptyset$} \Comment{No solutions with $l_i\ne l_j$.}
\EndIf
\If {$q$ contains two or more terms}
\State Add $q$ to $N$.
\EndIf
\EndFor
\State $S:=\emptyset$
\For {each $s$ in \Call{SolveInPowers}{$E,\ N$}} \Comment{$s$ is a map from $Y$ to linear polynomials in $Y$}
\State Set $x_i := 2^{s[y_i]}$ for each $i\in\{1,2,\dots,m\}$.
\State Solve $ML = X$ for $L$ composed of pairwise distinct integers and add the solution to $S$.   \Comment{Use Theorem~\ref{th:distinct} if needed.}
\EndFor
\State \Return {$S$}
\EndFunction
\end{footnotesize}
\end{algorithmic}
\end{algorithm}

\section{Solving a system of (in)equations in powers of $2$}\label{sec:solvep2}

For given finite sets $E$ and $N$ of nonzero linear polynomials in $x_1,x_2,\dots,x_m$ with integer coefficients, our goal is to find all tuples of nonnegative integers $(y_1,y_2,\dots,y_m)$ such that 
\[
\begin{split}
\forall p\in E:&\quad p(2^{y_1},2^{y_2},\dots,2^{y_m})=0, \\
\forall p\in N:&\quad p(2^{y_1},2^{y_2},\dots,2^{y_m})\ne 0.
\end{split}
\]

As simple as it sounds, the following theorem provides a foundation for our algorithm.

\begin{theorem}\label{th:eqval}
In any nonempty multiset of nonzero integers summing to $0$, there exist two elements with equal $2$-adic valuations.\footnote{Recall 
that the $2$-adic valuation of a nonzero integer $k$, denoted $\nu_2(k)$, equals the exponent of $2$ in the prime factorization of $k$, while $\nu_2(0)=\infty$.}
\end{theorem}

\begin{proof}
Let $S$ be a nonempty multiset of nonzero integers summing to $0$, and let $k$ be an element of $S$ with the smallest $2$-adic valuation $q:=\nu_2(k)$.
If all the other elements of $S$ have $2$-adic valuation greater than $q$, then the sum of all elements (which is $0$) has $2$-adic valuation $q$, which is impossible.
Hence, there exist at least two elements in $S$ that have $2$-adic valuation equal to $q$.
\end{proof}

Applying Theorem~\ref{th:eqval} to an equation $c_1x_1+\dots+c_mx_m\in E$, we conclude that
if only one of the coefficients $c_1,c_2,\dots,c_m$ is nonzero, then the system $(E,N)$ is insoluble. Otherwise, 
if there are two or more nonzero coefficients among $c_1,c_2,\dots,c_m$, then there exists a pair of indices $i<j$ such that $c_i\ne 0$, $c_j\ne 0$, and $\nu_2(c_ix_i) = \nu_2(c_jx_j)$, which implies
that we can make a substitution $x_i = 2^{\nu_2(c_j)-\nu_2(c_i)}x_j$ or $x_j = 2^{\nu_2(c_i)-\nu_2(c_j)}x_i$ (we pick one with integer coefficients). 
Then we proceed with making this substitution in $E$ and $N$, thus reducing the number of indeterminates. If this substitution does not nullify any polynomial in $N$, we proceed to solve the reduced system recursively. 
After exploring the pair $(i,j)$, we add a new inequation $2^{\nu_2(c_i)}x_i - 2^{\nu_2(c_j)}x_j$ to $N$
to avoid following the same substitutions in a different order and obtaining the same solutions in future, and proceed to the next pair of indices.

We outline the above description in Algorithm~\ref{alg:SolveInPowers}. For given sets $E$ and $N$ of linear equations and inequations in $x_1,\dots,x_m$ with integer coefficients, the
function \textsc{SolveInPowers}$(E,\ N)$ computes the set of their solutions in powers of $2$. 
Each solution is given in the form of a map $s$ from the set of variables $Y:=\{y_1,y_2,\dots,y_m\}$ to linear polynomials in these variables, representing the exponents in the powers of $2$.
Namely, $s$ sends every variable from $Y$ either to itself (when it is a free variable), or to a linear polynomial in the free variables.
For example, the map $\{ y_1 \to y_4+1,\ y_2\to y_2,\ y_3\to y_2+3,\ y_4\to y_4\}$ 
corresponds to the solution $(x_1,x_2,x_3) = (2^{y_4+1}, 2^{y_2}, 2^{y_2+3}, 2^{y_4})$, where $y_2$ and $y_4$ are free variables taking arbitrary nonnegative integer values.

\begin{algorithm}[!t]
\caption{An algorithm for solving a given system of linear equations $E$ and inequations $N$ over indeterminates from $X:=\{x_1,\dots,x_m\}$ in powers of $2$. It returns a set of maps $s$ from
variables $Y:=\{y_1,y_2,\dots,y_m\}$ to linear polynomials in these variables such that $(x_1,\dots,x_m) = (2^{s[y_1]},2^{s[y_2]},\dots,2^{s[y_m]})$ is a solution.}
\label{alg:SolveInPowers}
\begin{algorithmic}[1]
\begin{footnotesize}
\Function{SolveInPowers}{$E,\ N$}
\If {$E=\emptyset$}
\State \Return {$\{\text{the identity map}\}$}  \Comment{Every variable in $Y$ is free.}
\EndIf
\State Pick $c_1x_1+\dots+c_mx_m\in E$ with the smallest number of nonzero coefficients.
\State Let $I := \{ i\mid 1\leq i\leq m,\ c_i\ne 0\}$ be the set of indices of nonzero coefficients.
\State Set $S:=\emptyset$   \Comment{We accumulate solutions in $S$.}
\For {each $\{i,j\}\subseteq I$} \Comment{We iterate over all 2-element subsets of $I$.}
\State Possibly exchanging the values of $i$ and $j$, ensure that $d := \nu_2(c_i)-\nu_2(c_j)\geq 0$.
\State Compute $N'$ from $N$ by substituting $x_j \leftarrow 2^d x_i$ and excluding nonzero constant polynomials.
\If {$0\in N'$}
\State \textbf{continue} to the next pair $\{i,j\}$.
\EndIf
\State Add $x_j - 2^d x_i$ to $N$. \Comment{For future we disallow the equality $c_jx_j = c_i x_i$.}
\State Compute $E'$ from $E$ by substituting $x_j \leftarrow 2^d x_i$ and excluding zero polynomials.
\For {each $s$ in \Call{SolveInPowers}{$E',\ N'$}}  \Comment{$s$ is a map from $Y$ to linear polynomials in $Y$.}
\State Redefine $s[y_j] := s[y_i] + d$.
\State Add $s$ to set $S$.
\EndFor
\EndFor
\State \Return {$S$}
\EndFunction
\end{footnotesize}
\end{algorithmic}
\end{algorithm}

We note that Algorithm~\ref{alg:SolveInPowers} is a recursive branch-and-bound algorithm with a maximum recursion depth equal to $m-1$. In the worst case without bounding, at depth $d$ it would branch into $\binom{m-d}{2}=\frac{(m-d)(m-d-1)}2$ recursive calls, and thus its complexity can be estimated as $O\big((|E|+|N|)\cdot\frac{m!^2}{2^m}\big)$ algebraic operations. However, this naive estimate does not take into account the bounding controlled by the elements of $N$ and the restricted branching controlled by the choice of sparse elements of $E$. Although it is difficult to give an accurate complexity estimate, our computational results presented in Sections~\ref{sec:MFS}--\ref{sec:MAG} show that Algorithm~\ref{alg:SolveInPowers} and thus Algorithm~\ref{alg:GraphSolve} are quite efficient in practice 
(in particular, see Tables~\ref{tab:statMFS}, \ref{tab:statINAD}, \ref{tab:statMAG}).

\begin{figure}[!t]
\centering \includegraphics[width=.8\linewidth]{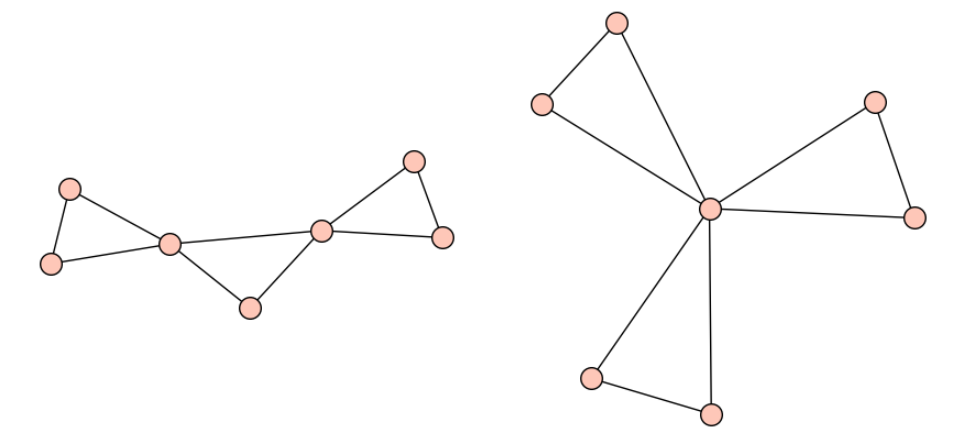}
\caption{Minimal forbidden subgraphs of order $7$.}
\label{fig:mfs7}
\end{figure}

\section{Minimal forbidden subgraphs}\label{sec:MFS}

We use the proposed algorithms to find minimal forbidden subgraphs of small order. 
It is easy to see that each MFS must be connected. Also, if an inadmissible graph contains a vertex of degree 1, then the graph without this vertex is also inadmissible. Therefore, the minimum degree of an MFS must be $\geq 2$.

It is almost trivial to verify that $C_4$ is the smallest MFS and the only one on $4$ vertices. 
Therefore, for $n>4$ we can restrict our attention to the connected squarefree graphs with the minimum degree $\geq 2$ as candidates, which we generate in SageMath~\cite{Sage} with the function \texttt{nauty\_geng()} based on \texttt{nauty} tool~\cite{nauty} 
supporting both connected (option \texttt{-c}) and squarefree (option \texttt{-f}) graphs as well as a lower bound on the minimum degree (option \texttt{-d\#}). This significantly speeds up the algorithm and eliminates the need to test the presence of MFS $C_4$ as a subgraph.

We search for MFSs, other than $C_4$ iteratively increasing their order, accumulating found MFSs in a set $S$ (initially empty). 
For each order, we iterate over the candidate graphs $G$ in the order of non-decreasing size, and check if $G$ contains any graphs from $S$ as a subgraph using the SageMath function \texttt{is\_subgraph()}. 
If $G$ contains any of the graphs from $S$, we go to the next candidate graph $G$.
Otherwise, we test the admissibility of $G$ by calling $\textsc{GraphSolve}(G)$. 
If $G$ is inadmissible, then it represents an MFS and we add it to $S$.
The described algorithm is outlined in Algorithm~\ref{alg:MFS}. 

\begin{algorithm}[!t]
\caption{An algorithm for iterative computing minimal forbidden subgraphs, other than $C_4$, of order up to $u$.}
\label{alg:MFS}
\begin{algorithmic}[1]
\begin{footnotesize}
\Function{FindMFS}{$u$}
\State $S := \emptyset$
\For {$n=5,\dots,u$}
\For {each connected squarefree graph $G$ of order $n$ and minimum degree $\geq 2$, iterated over such that the size of $G$ is nondecreasing}
\If {$G$ contains any graph from $S$ as a subgraph}
\State \textbf{continue} to next $G$
\EndIf
\If {\Call{GraphSolve}{$G$} is empty}
\State Add $G$ to the set $S$.
\EndIf
\EndFor
\EndFor
\State \Return {$S$}
\EndFunction
\end{footnotesize}
\end{algorithmic}
\end{algorithm}

We determine that the next smallest MFSs after $C_4$ are two graphs of order $7$ (Fig.~\ref{fig:mfs7}).
One of these graphs was previously shown to be inadmissible by Bolan while proving that $g(10)=15$~\cite{Bolan2022}.

\begin{figure}[!t]
\centering \includegraphics[width=1.0\linewidth]{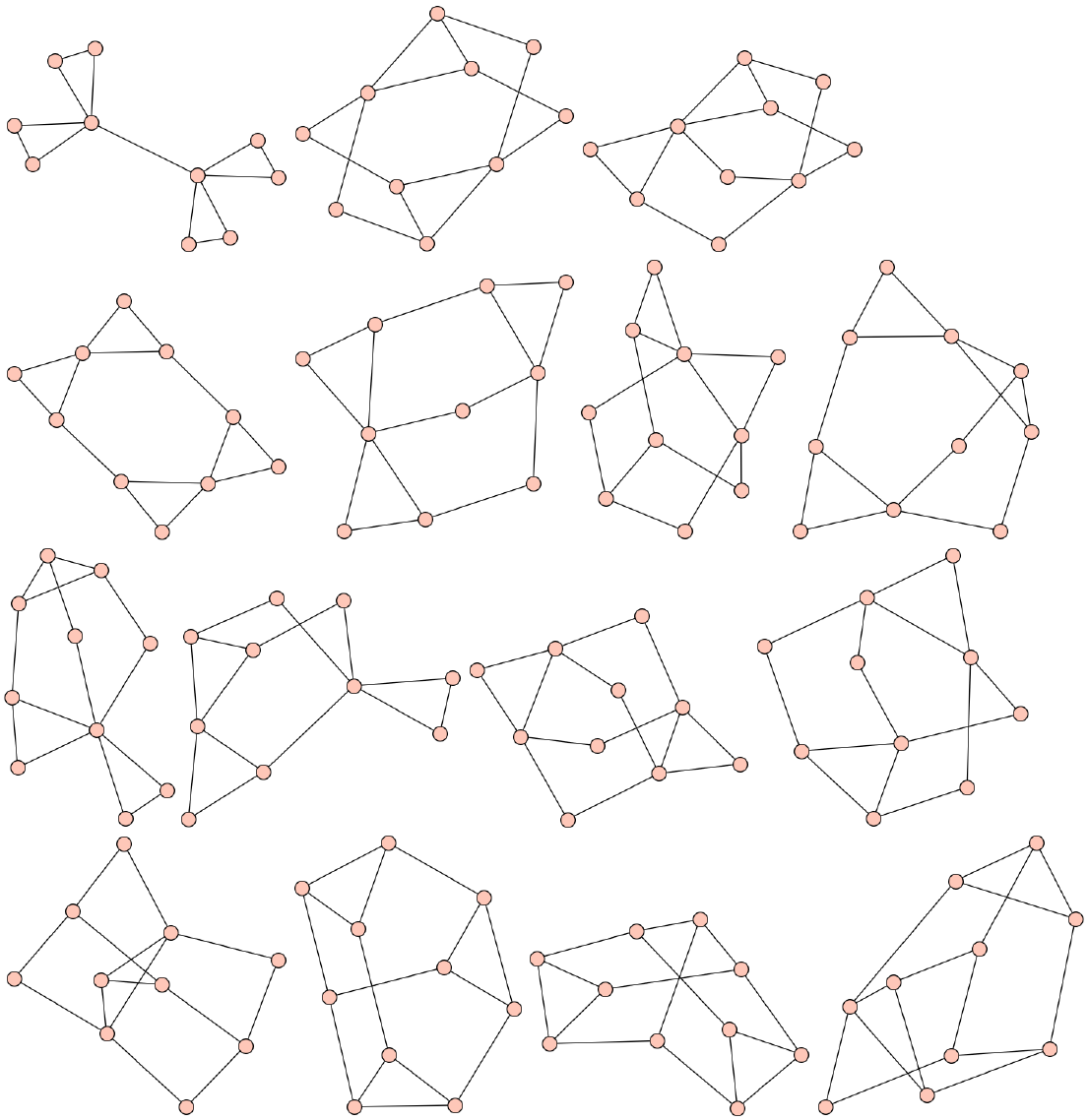}
\caption{Minimal forbidden subgraphs of order $10$.}
\label{fig:mfs10}
\end{figure}

There are no MFSs of order $8$ or $9$, but there are $15$ of them of order $10$ (Fig.~\ref{fig:mfs10}), and there are $77$ MFSs of order $11$.
We evaluate the performance\footnote{Here and below the performance is benchmarked at
a desktop computer with Intel Xeon E5-2670v2 2.50GHz CPU.}
of this computation in Table~\ref{tab:statMFS}.  
We use MFSs of order $\leq 10$ for quick filtering of some inadmissible graphs.

\begin{table}[!tb]
\begin{center}
\begin{tabular}{|c||c|c|c|c|c|c|c|}
\hline
Graph order & 5 & 6 & 7 & 8 & 9 & 10 & 11 \\
\hline\hline
\# candidate graphs & 2 & 3 & 10 & 28 & 112 & 533 & 3126 \\ 
\hline
\# graphs with an MFS & 0 & 0 & 0 & 1 & 12 & 64 & 528 \\ 
\hline
\# graphs tested w. $\textsc{GraphSolve}$ & 2 & 3 & 10 & 27 & 100 & 469 & 2598 \\ 
\hline
Average test time (sec.) & 1.0 & 1.1 & 1.2 & 1.7 & 2.3 & 3.5 & 2.4 \\
\hline
\end{tabular}
\end{center}
\caption{Performance benchmarks for computing MFSs of orders $5$ to $11$.}\label{tab:statMFS}
\end{table}

\section{Computing values of $g(n)$}\label{sec:gofn}

We use the proposed algorithms for computing values of $g(n)$ (sequence \texttt{A352178} in the OEIS) for $n$ in $[12,21]$, relying on the known values $g(n)$ for $n\leq 11$ and the lower 
bound $\ell(n)$ (sequence {\tt A347301} in the OEIS) for $n\geq 12$:
\begin{center}
\begin{footnotesize}
\begin{tabular}{c}
\begin{tabular}{|c|c|c|c|c|c|c|c|c|c|c|c|}
\hline
$n$ & 1 & 2 & 3 & 4 & 5 & 6 & 7 & 8 & 9 & 10 & 11 \\
\hline
$g(n)$ & 0 & 1 & 3 & 4 & 6 & 7 & 9 & 11 & 13 & 15 & 17 \\
\hline
\end{tabular}
\\
\\
\begin{tabular}{|c|c|c|c|c|c|c|c|c|c|c|}
\hline
$n$ & 12 & 13 & 14 & 15 & 16 & 17 & 18 & 19 & 20 & 21 \\
\hline
$\ell(n)$ & 19 & 21 & 24 & 26 & 29 & 31 & 34 & 36 & 39 & 41 \\
\hline
\end{tabular}
\end{tabular}
\end{footnotesize}
\end{center}
We will establish the following theorem:

\begin{theorem}\label{th:gell}
For each integer $n\in[12,21]$, we have $g(n)=\ell(n)$.
\end{theorem}

The remainder of this section is devoted to the proof of Theorem~\ref{th:gell}.
Our goal is to show that for each $n\in[12,21]$ we have $g(n)\leq \ell(n)$, which is equivalent to showing that every graph of order $n$ and size $\ell(n)+1$ is inadmissible.

\begin{theorem}\label{th:gbound}
For any integer $n>2$:
\begin{itemize}
\item[\rm (I)]
if an admissible graph of order $n$ and size $e$ exists, then its minimum degree is at least $e-g(n-1)$;

\item[\rm (II)]
$$g(n) \leq \left\lfloor \frac{n\cdot g(n-1)}{n-2}\right\rfloor.$$

\item[\rm (III)] if $$\frac{g(n-1)}{n-1} = \max_{k\in\{1,2,\dots,n-1\}} \frac{g(k)}k,$$ then
any admissible graph of order $n$ and size $e>\frac{n}{n-1}g(n-1)$ is connected.
\end{itemize}
\end{theorem}

\begin{proof}
Let $G$ be an admissible graph of order $n$ with $e$ edges. If $G$ has a vertex of degree smaller than $e-g(n-1)$, then removing it from $G$ results in an admissible graph of order $n-1$ with more than $g(n-1)$ edges. 
The contradiction shows that the degree of any vertex of $G$ is at least $e-g(n-1)$, proving the statement (I).
Furthermore, it implies that $G$ has at least $\tfrac{n(e-g(n-1))}2$ edges, that is, $e\geq \tfrac{n(e-g(n-1))}2$, implying that $e \leq \big\lfloor\frac{n\cdot g(n-1)}{n-2}\big\rfloor$. For an admissible graph of order $n$ with $e=g(n)$ edges, it implies $g(n) \leq \big\lfloor\tfrac{n\cdot g(n-1)}{n-2}\big\rfloor$, which proves the statement (II).

To prove the statement (III), suppose that we have an admissible graph $G$ of order $n$ and size $e>\frac{n}{n-1}g(n-1)$. Let $t$ be the number of its connected components, and let $s_1,\dots,s_t$ be their orders. Let us assume that $t\geq 2$, implying that $\max_{1\leq i\leq t} s_i \leq n-1$.
Clearly, we have $s_1+\dots+s_t = n$ and $e\leq g(s_1)+\dots+g(s_t)$. It follows that
$$\frac{g(s_1)+\dots+g(s_t)}{s_1+\dots+s_t}\geq \frac{e}{n} > \frac{g(n-1)}{n-1}.$$
On the other hand, since $\frac{g(s_1)+\dots+g(s_t)}{s_1+\dots+s_t}$ is the mediant fraction of $\frac{g(s_1)}{s_1}, \dots, \frac{g(s_t)}{s_t}$, we have
$$
\frac{g(s_1)+\dots+g(s_t)}{s_1+\dots+s_t} \leq \max_{1\leq i\leq t} \frac{g(s_i)}{s_i} \leq \max_{k\in\{1,2,\dots,n-1\}} \frac{g(k)}k = \frac{g(n-1)}{n-1}.
$$
The contradiction proves that we cannot have $t\geq 2$, i.e., $G$ is connected.
\end{proof}

Iteratively for each $n\in[12,21]$, Theorem~\ref{th:gbound}(III) from the established values $g(k)$ for all $k<n$ implies that any admissible graph of order $n$ and size $\geq\ell(n)$ is connected.
That is, in the venue of proving Theorem~\ref{th:gell} we can focus on the connected squarefree graphs only, which we refer to as \emph{candidate graphs}. 
We test the admissibility of a candidate graph $G$ by first checking the presence of any MFS of order $\leq 10$ as a subgraph in $G$, and then, if no MFS is present, by invoking the {\sc GraphSolve} algorithm for $G$. 

For order $n=12$, Theorem~\ref{th:gbound}(I) implies that if there is an admissible graph of size $\ell(12)+1=20$, then its minimum degree should be at least $20-g(11)=3$. Hence, to prove Theorem~\ref{th:gell} for $n=12$, we generate all candidate graphs of order $12$, size $20$, and minimum degree $\geq 3$, which we then test for admissibility.
There are $18$ candidate graphs and none are admissible, which proves $g(12)=19$. Similarly, for $n=13$, we generate all candidate graphs of order $13$, size $\ell(13)+1=22$, and minimum degree $\geq 22-g(12)=3$. There are $173$ such candidate graphs, and again none of them are admissible.

From $g(13)=21$, Theorem~\ref{th:gbound}(II) implies that $g(14)\leq 24$, which matches the lower bound. Therefore, we obtain $g(14)=24$ without any computation.

For order $n=15$, we test if there is any admissible graph of size $\ell(15)+1=27$. By Theorem~\ref{th:gbound}(I) such a graph should have the minimum degree $\geq 3$.
We generate $8,280$ such candidate graphs, but our check shows that none are admissible. Hence, $g(15)=26$. 

\begin{table}[!tb]
\begin{center}
\begin{tabular}{|c||c|c|c|c|}
\hline
Graph order & 12 & 13 & 15 & 20 \\
\hline\hline
\# candidate graphs & 18 & 173 & 8280 & 15,156 \\ 
\hline
\# graphs with an MFS & 16 & 159 & 8252 & 14,591 \\ 
\hline
\# graphs tested w. $\textsc{GraphSolve}$ & 2 & 14 & 28 & 565 \\ 
\hline
Average test time (sec.) & 6.5 & 17.5 & 24.5 & 247.9 \\
\hline
\end{tabular}
\end{center}
\caption{Performance benchmarks for testing (inadmissible) graphs of order $n$, size $\ell(n)+1$, and minimum degree $\geq 3$, for $n\in\{12,13,15,20\}$.}\label{tab:statINAD}
\end{table}

For order $n=16$, Theorem~\ref{th:gbound}(II) implies $g(16)\leq 29=\ell(16)$ and thus $g(16)=29$.

For order $n=17$, Theorem~\ref{th:gbound}(I) implies that any admissible graph of size $\ell(17)+1=32$ should have the minimum degree $\geq 3$. 
In fact, the minimum degree should be exactly $3$, since otherwise the size would be at least $17\cdot 4/2=34$.
There are $1,023,100$ such candidate graphs, which, in principle, are possible to inspect directly, although it would be quite time consuming.
Instead, we approached this problem from another angle: noticing that the removal of a vertex of degree $3$ from an admissible graph of order $17$ and size $32$ results 
in a maximum admissible graph (MAG) of order $16$. We constructed all MAGs of order $16$ as explained in Section~\ref{sec:MAG} and established that neither one extends 
to an admissible graph of order $17$ and size $32$, thus proving that $g(17)=\ell(17)=31$.

For order $n=18$, Theorem~\ref{th:gbound}(II) implies $g(18)\leq 34=\ell(18)$ and thus $g(18)=34$.

For order $n=19$, Theorem~\ref{th:gbound}(I) implies that any admissible graph of size $\ell(19)+1=37$ should have the minimum degree $\geq 3$. 
With necessity such a graph has a vertex of degree $3$ and we proceed similarly to the case $n=17$ above. 
We construct MAGs of order $18$ (see Section~\ref{sec:MAG}) and show that none of them can be extended to an admissible graph of order $19$ and size $37$, implying that $g(19)=\ell(19)=36$.

For order $n=20$, Theorem~\ref{th:gbound}(I) implies that any admissible graph of size $\ell(20)+1=40$ should have the minimum degree $\geq 4$, from where it follows that this graph is regular of degree $4$. There are $15,156$ such candidate graphs, all of which are inadmissible. Hence, $g(20)=\ell(20)=39$.

For order $n=21$, Theorem~\ref{th:gbound}(I) implies that any admissible graph of size $\ell(21)+1=42$ should have the minimum degree $\geq 3$. 
First, we construct MAGs of order $20$ and show that none of them can be extended to an admissible graph by adding a vertex of degree $3$. Second, we note a graph of order $21$, size $42$, and minimum degree $4$ is a $4$-regular graph. Let's assume that such an admissible graph $G$ exists. Since $G$ does not have a cycle of length $4$, its girth is either $3$ or $\geq 5$. If the girth of $G$ is $3$, then removing from $G$ vertices forming a $3$-cycle results in an admissible graph $G'$ of order $18$ and size $31$, composed of $3$ vertices of degree $2$ and $15$ vertices of degree $4$. Further removing from $G'$ two degree-$2$ vertices we obtain a MAG of order $16$, which however cannot contain a vertex of degree $2$ (Fig.~\ref{fig:msg16}). Hence, $G$ has girth $\geq 5$. There exist $8$ such graphs as computed by tool {\tt GENREG}~\cite{Meringer1999}, which are listed at~\cite{web_girth5}. However, our test shows that they all are inadmissible and therefore $G$ cannot exist, implying that $g(21)=\ell(21)=41$.

This concludes our proof of Theorem~\ref{th:gell}.
We evaluate the performance of the computations used in this proof in Table~\ref{tab:statINAD}.

\section{Maximum admissible graphs}\label{sec:MAG}

\begin{table}[!tb]
\begin{center}
\begin{tabular}{|c||c|c|c|c|c|}
\hline
Graph order & 14 & 15 & 16 & 17 & 18  \\
\hline\hline
\# candidate graphs & 2,184 & 33,732 & 243 & 5,847,788 & 287 \\ 
\hline
\# graphs with an MFS & 1,976 & 29,251 & 215 & 5,734,238 & 257 \\ 
\hline
\# graphs tested w. {\sc GraphSolve} & 208 & 4,481 & 28 & 113,550 & 30 \\ 
\hline
Average test time (sec.) & 30.9 & 44.2 & 51.0 & 81.4 & 79.1 \\
\hline
\end{tabular}
\end{center}
\caption{Performance benchmarks for identifying maximum admissible graphs.} %of order $n\in\{14,15,16,17,18\}$.}
\label{tab:statMAG}
\end{table}

\begin{figure}[!t]
\begin{tabular}{cc}
\includegraphics[width=.45\linewidth]{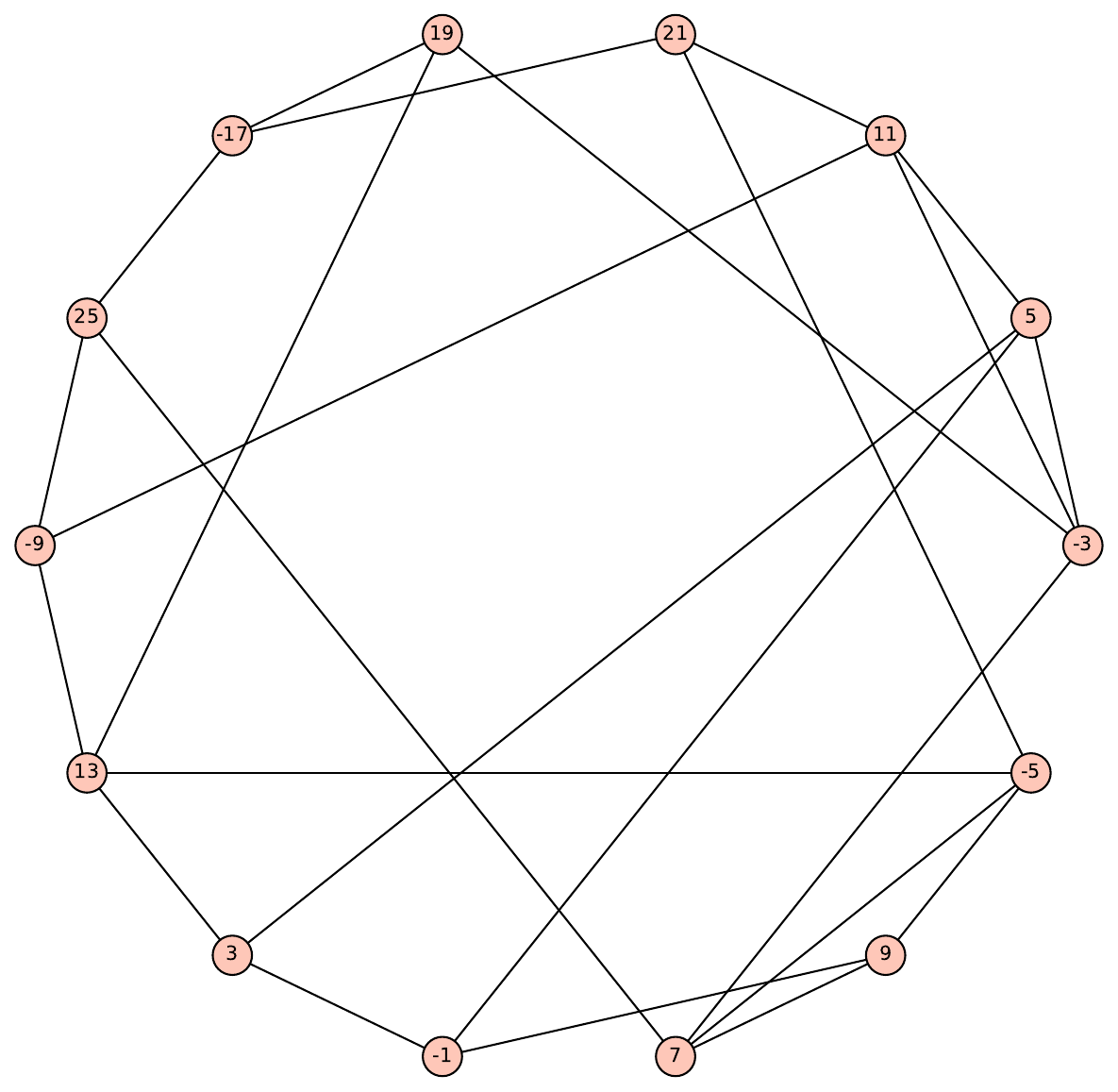} & 
\includegraphics[width=.45\linewidth]{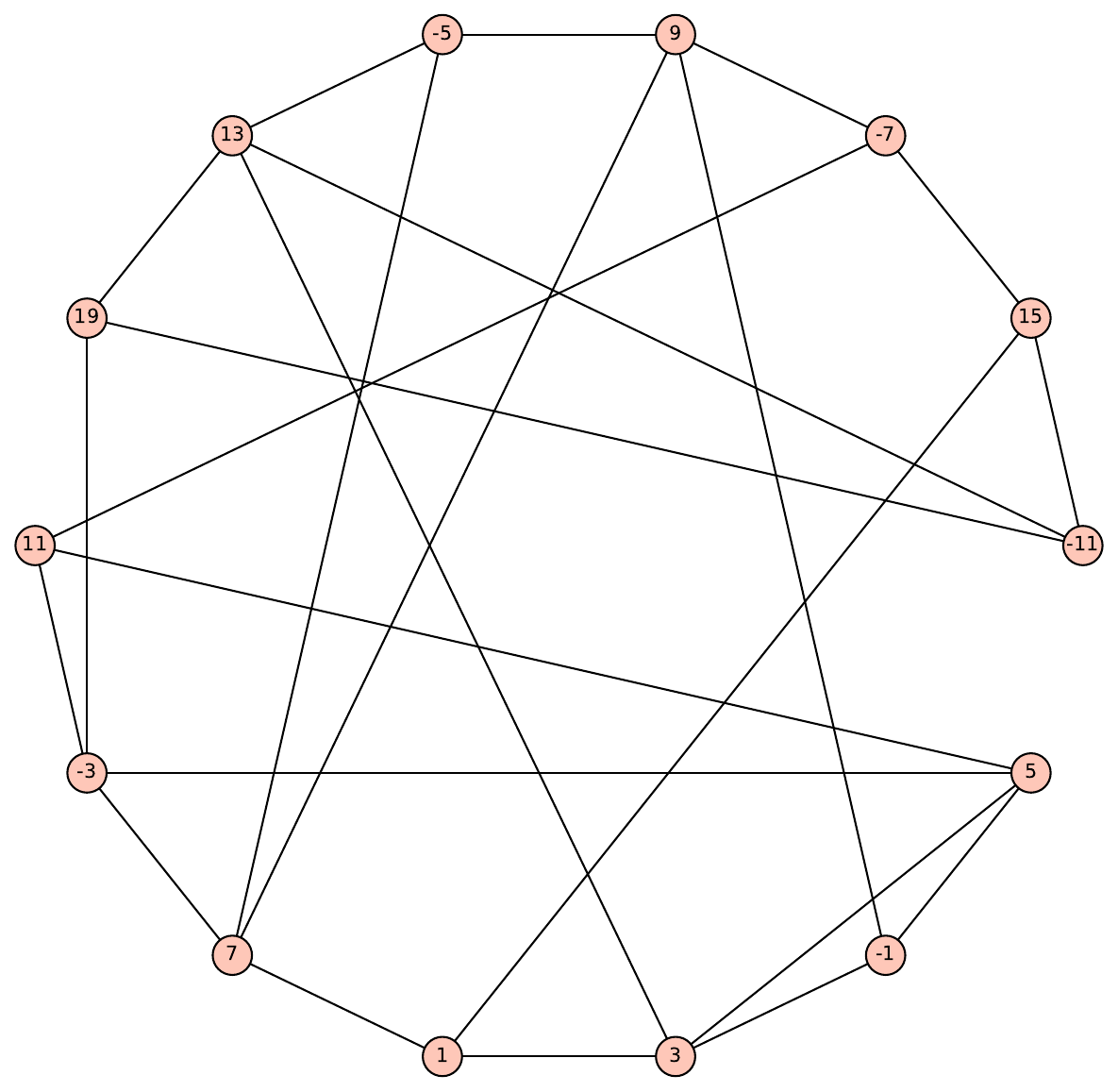} \\
\includegraphics[width=.45\linewidth]{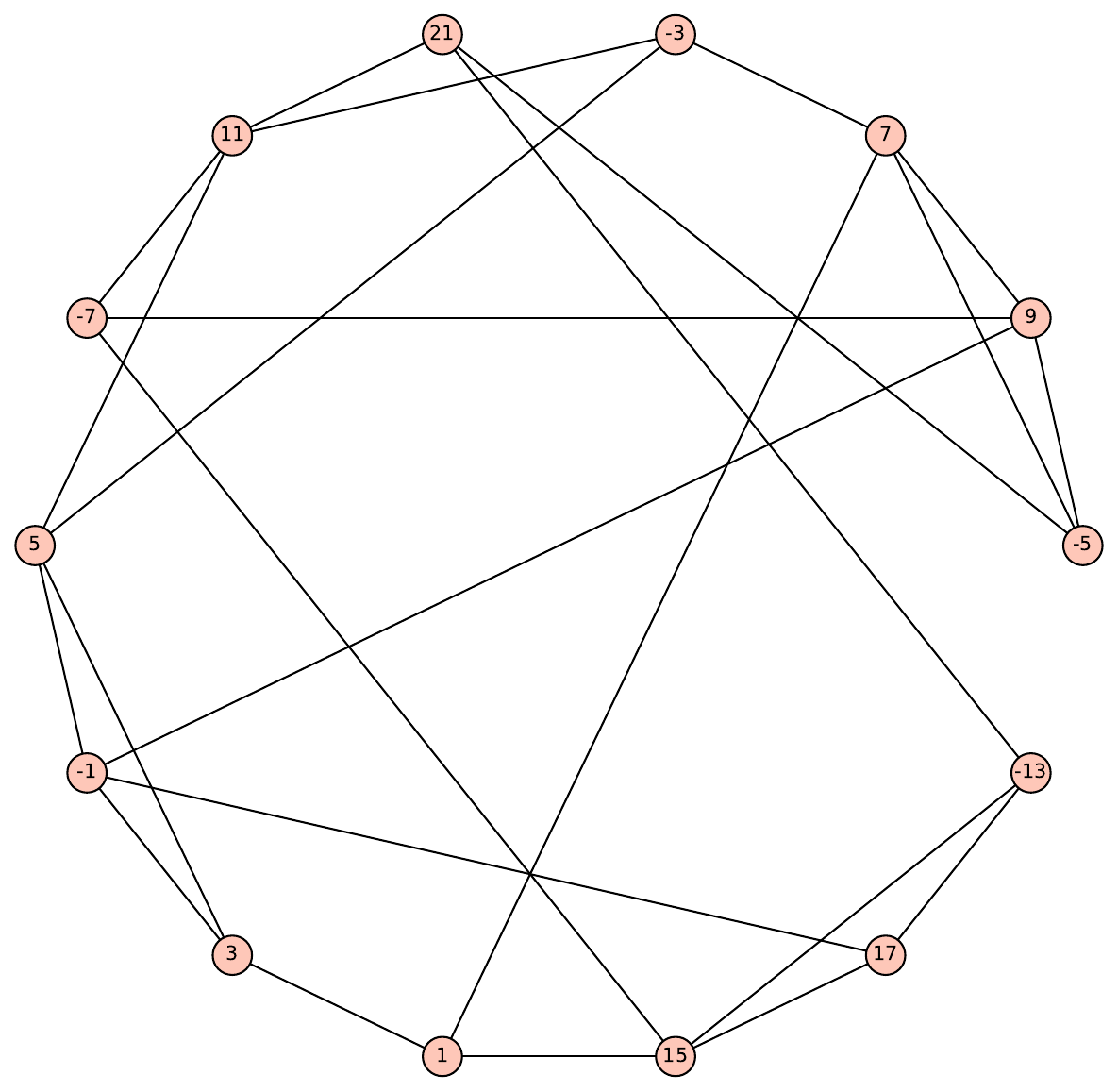} & 
\includegraphics[width=.45\linewidth]{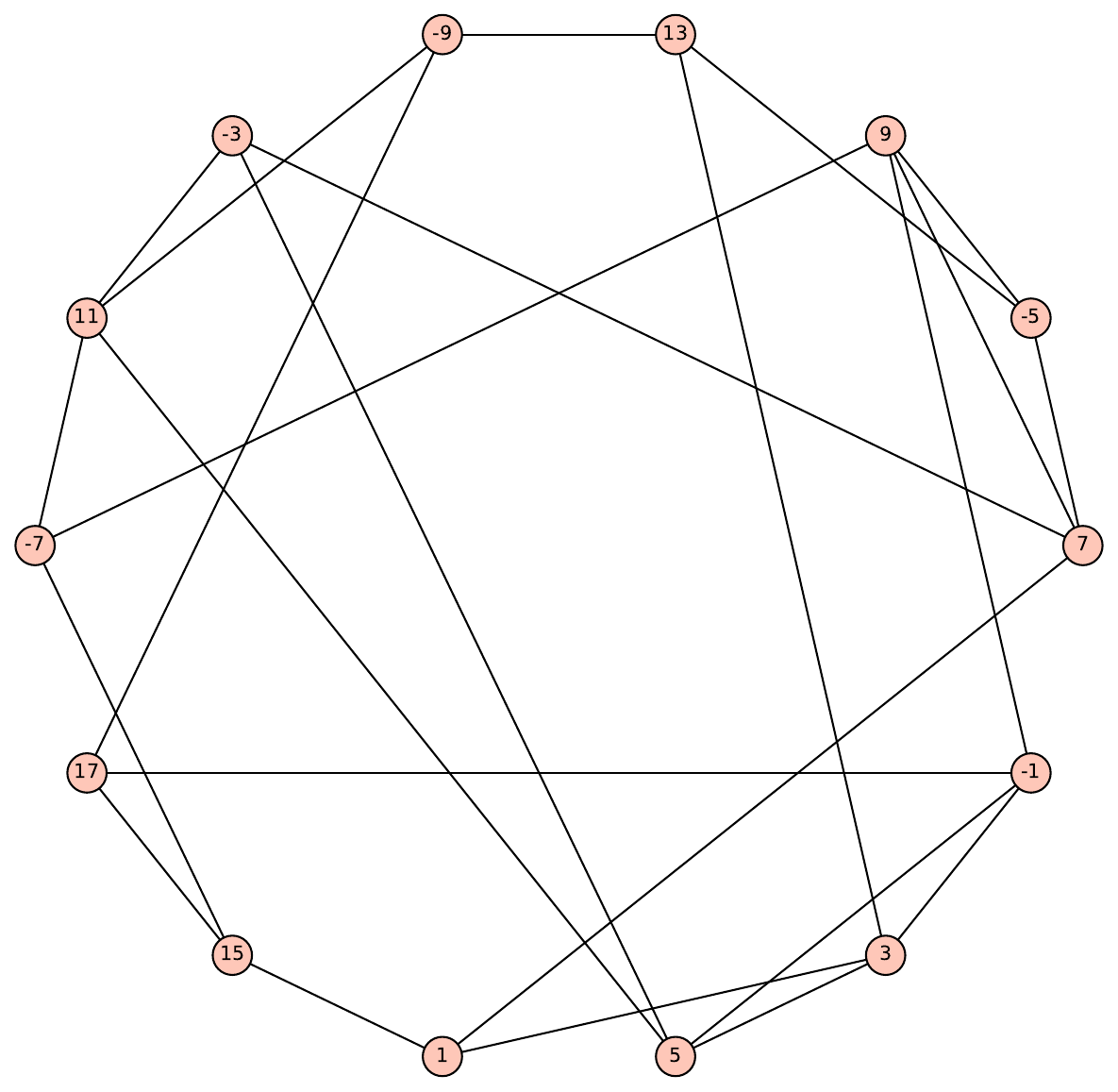}
\end{tabular}
\caption{Maximum admissible graphs of order $14$ with the corresponding vertex labelings.}
\label{fig:msg14}
\end{figure}

From the definition of $g(n)$, it follows that the maximum admissible graphs (MAG) of order $n$ have size $g(n)$ each.
All MAGs of each order $n\leq 14$ can be obtained directly from the candidate graphs generated by \texttt{nauty}.
In particular, for $n=14$ we can restrict our attention to the $2,184$ candidate graphs of minimum degree $3$, among which we identified only $4$ admissible graphs (Fig.~\ref{fig:msg14}).

To construct MAGs of order $n\in\{15,16,17,18\}$, which by Theorem~\ref{th:gbound}(I) have the minimum degree $\geq g(n)-g(n-1)$, we generate and test for admissibility candidate graphs of two types: 
\begin{itemize}
    \item \emph{extended graphs} resulted from adding a vertex of degree $g(n)-g(n-1)$ to a MAG of order $n-1$; and 
    \item \emph{denovo graphs} with the minimum degree $\geq g(n)-g(n-1)+1$ generated by \texttt{nauty}.
\end{itemize}
Clearly, denovo candidate graphs may exist only if $n(g(n)-g(n-1)+1)\leq 2g(n)$, i.e., $g(n)\leq \frac{n}{n-2}(g(n-1)-1)$.

For $n=15$, 
there are $124$ extended and $33,608$ denovo candidate graphs, among which we identified $20$ and $8$ MAGs, respectively.

For $n=16$, there are $243$ extended and no denovo candidate graphs, delivering just two MAGs (Fig.~\ref{fig:msg16}). 
We use these MAGs to establish $g(17)=31$ as explained in Section~\ref{sec:gofn}.

For $n=17$, there are $82$ extended and $5,847,706$ denovo candidate graphs, delivering $11$ and $7$ MAGs, respectively.

For $n=18$, there are $287$ extended and no denovo candidate graphs, delivering just two MAGs.
We use these MAGs to establish $g(19)=36$ as explained in Section~\ref{sec:gofn}.
Benchmarks of these computations 
are summarized in Table~\ref{tab:statMAG}.

For MAGs of order $19$, we obtain $15$ MAGs of minimum degree $2$ from the extended graphs. However, the number of denovo candidate graphs appears to be prohibitively large. We therefore take a different route by first computing \emph{sub-maximum addmissible graphs} (sub-MAGs) that have size just one less than that of MAGs. Namely, we compute:
\begin{itemize}
    \item sub-MAGs of order $14$ composed of $400$ graphs of minimum degree $2$ extended from MAGs of order $13$ and $34$ denovo graphs of minimum degree $3$;
    \item sub-MAGs of order $15$ composed of $54$ graphs of minimum degree $1$ extended from MAGs of order $14$, $1597$ graphs of minimum degree $2$ extended from sub-MAGs of order $14$, and $70$ denovo graphs of minimum degree $1$;
    \item sub-MAGs of order $16$ composed of $144$ graphs of minimum degree $2$ extended from MAGs of order $15$ and $36$ denovo graphs of minimum degree $3$;
    \item sub-MAGs of order $17$ composed of $32$ graphs of minimum degree $1$ extended from MAGs of order $16$, $909$ graphs of minimum degree $2$ extended from sub-MAGs of order $16$, $0$ denovo graphs of minimum degree $3$ without a connected pair of degree-$3$ vertices,\footnote{Here any denovo graph has $\geq 4\cdot 17 - 2\cdot 30 = 8$ independent degree-$3$ vertices, whose removal results in a graph of order $9$ and size $6$. We generate all $50$ such graphs, and then extend them with $8$ independent degree-$3$ vertices to construct the candidate denovo graphs.} and $131$ graphs extended from sub-MAGs of order $15$ with a connected pair of degree-$3$ vertices;
    \item sub-MAGs of order $18$ composed of $124$ graphs of minimum degree $2$ extended from MAGs of order $17$ and $50$ graphs of minimum degree $3$ extended from sub-MAGs of order $17$.
\end{itemize}
We then construct MAGs of order $19$ and minimum degree $3$ by extending sub-MAGs of order $18$, resulting in $7$ such MAGs.

MAGs of order $20$ have the minimum degree $3$ and are the result of an extension of MAGs of order $19$, which gives $2$ such MAGs.
We use these MAGs to establish $g(21)=41$ as explained in Section~\ref{sec:gofn}.

\begin{figure}[!t]
\begin{tabular}{cc}
\includegraphics[width=.45\linewidth]{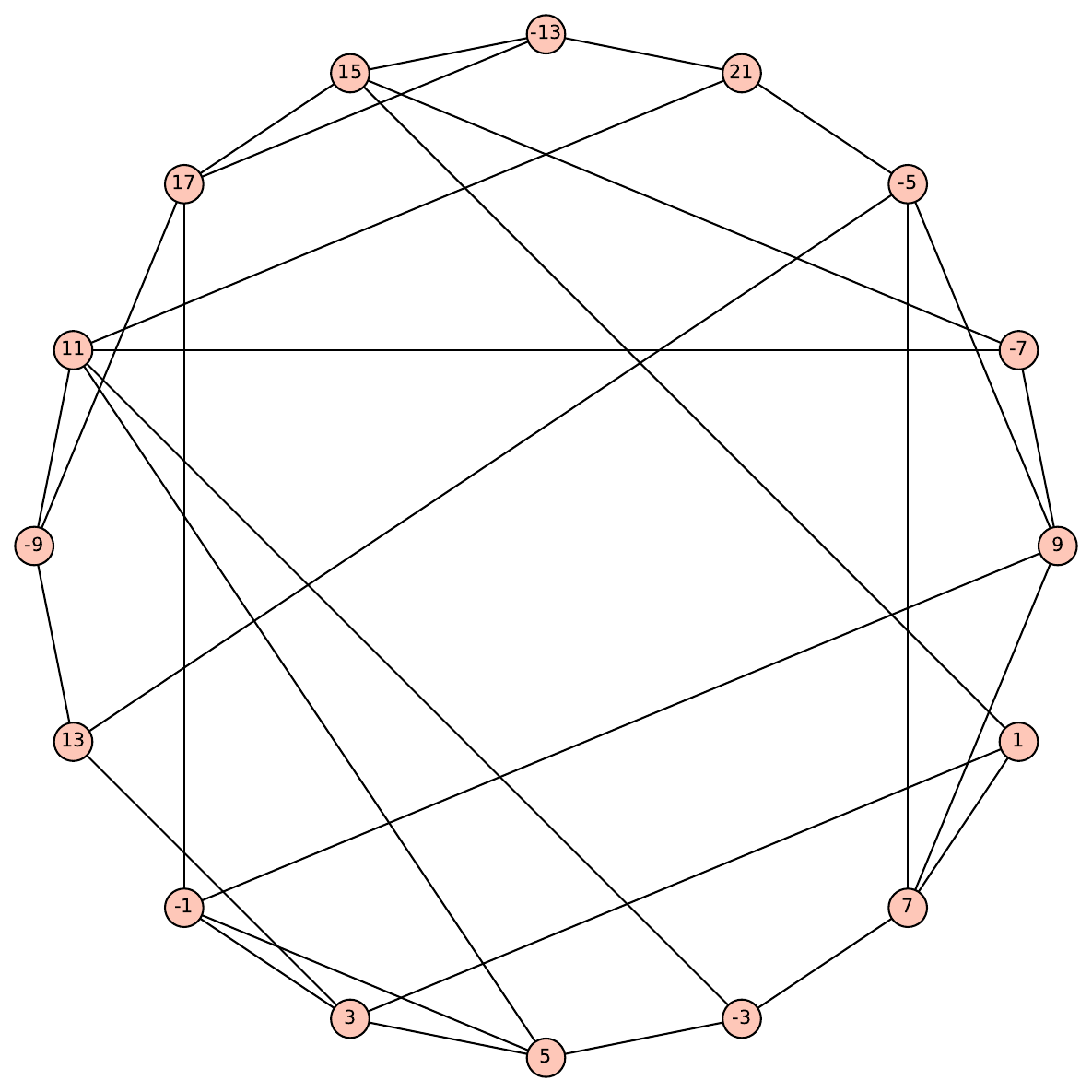} & 
\includegraphics[width=.45\linewidth]{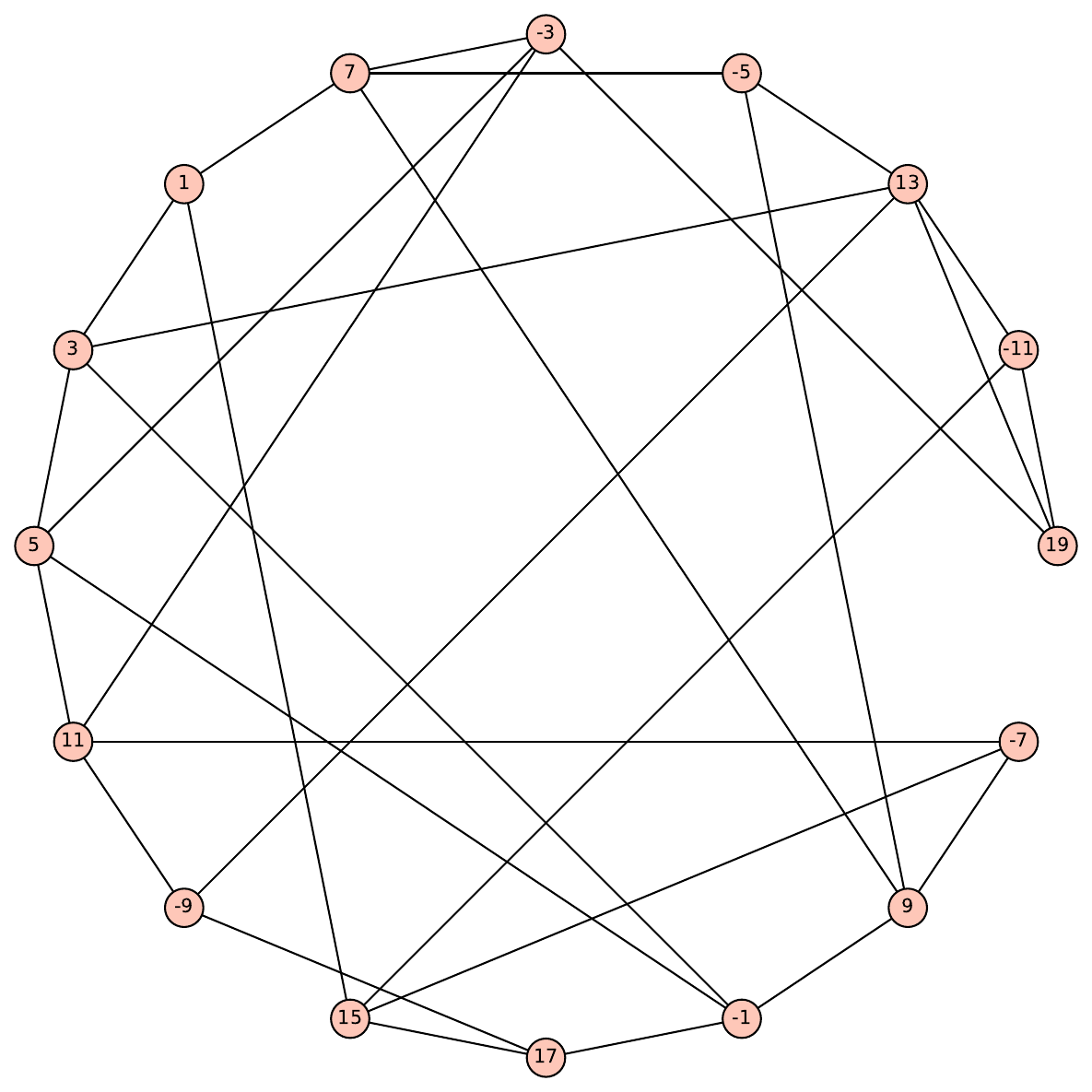}
\end{tabular}
\caption{Maximum admissible graphs of order $16$ with the corresponding vertex labelings.}
\label{fig:msg16}
\end{figure}

\section{Discussion}

In this work, we propose the algorithm {\sc SolveInPowers}~\ref{alg:SolveInPowers} to solve a given system of equations and inequations in powers of 2, 
and applied it to the problem of determining the admissibility of a given graph with respect to labeling its vertices with pairwise distinct integers summing to a power of 2 for each edge
(as outlined in the {\sc GraphSolve} algorithm~\ref{alg:GraphSolve}).
The algorithm {\sc SolveInPowers} can be easily extended to the powers of primes other than $2$, and further to the powers of an arbitrary positive integer $b$ by introducing independent
variables for the powers of each prime factor of $b$.
Furthermore, while we applied {\sc SolveInPowers} to \emph{linear} equations and inequations, 
the linearity appears to be inessential making it just a technical task to extend the algorithm to higher-degree polynomial equations and inequations.

The performance of {\sc SolveInPowers} is sensitive to the number of nonzero coefficients in the equations $E$, 
which in the underlying {\sc GraphSolve} algorithm~\ref{alg:GraphSolve} correspond to nonzero elements of $K_l$. 
Although $K_l$ can be any basis of the left kernel of the given graph's incidence matrix, we found that upfront LLL reduction of the basis helps reduce the number of nonzero elements in $K_l$ 
and noticeably improves performance.
We have also explored the idea of performing an LLL reduction of the equations in $E$ after each substitution, but it did not seem to provide much benefit.

Theorem~\ref{th:gbound} easily generalizes to graphs with a property that is invariant with respect to \emph{induced} subgraphs (i.e., if a graph has such a property, then so does each of its induced subgraphs).
In particular, if $h(n)$ denotes the maximum size of a squarefree graph on $n$ vertices (OEIS {\tt A006855}), then $h(n) \leq \left\lfloor h(n-1)\cdot n/(n-2)\right\rfloor$. 
At the same time, it is known that $h(n)$ grows proportionally to $n^{3/2}$ and satisfies $h(n) \leq \left\lfloor n(1+\sqrt{4n-3})/4\right\rfloor$~\cite{Reiman1958,Aigner2018}, which also gives an upper bound for $g(n)$.
Therefore, we cannot expect better asymptotic bounds from the inequality in Theorem~\ref{th:gbound}(II) alone.

An implementation of the proposed algorithms in SageMath~\cite{Sage} and the computed graphs data are available at~\url{https://github.com/maxale/power2_graphs}.

\section*{Acknowledgements}

The author thanks Neil Sloane for his nice introduction to the problem~\cite{Haran2022} and proofreading of the earlier version of this paper. An extended abstract of this work has been published in the proceedings of the \emph{$35^\text{th}$ International Workshop on Combinatorial Algorithms (IWOCA)}~\cite{Alekseyev2024IWOCA}.

\bibliographystyle{plainurl}
\bibliography{power2_graphs}

\begin{thebibliography}{10}

\bibitem{Aigner2018}
M.~Aigner and G.~M. Ziegler.
\newblock {\em {Proofs from THE BOOK}}.
\newblock Springer, Berlin Heidelberg, 6th edition, 2018.

\bibitem{Alekseyev2024IWOCA}
M.~A. Alekseyev.
\newblock On computing sets of integers with maximum number of pairs summing to
  powers of 2.
\newblock In A.~A. Rescigno and U.~Vaccaro, editors, {\em {Combinatorial
  Algorithms. IWOCA 2024}}, volume 14764 of {\em Lecture Notes in Computer
  Science}, pages 1--13, Cham, 2024. Springer Nature Switzerland.
\newblock \href {https://arxiv.org/abs/2303.02872} {\path{arXiv:2303.02872}},
  \href {https://doi.org/10.1007/978-3-031-63021-7_1}
  {\path{doi:10.1007/978-3-031-63021-7_1}}.

\bibitem{Haran2022}
{B. Haran and N. J. A. Sloane}.
\newblock {Problems with Powers of Two}.
\newblock {Numberphile Youtube Channel}, September 2022.
\newblock \url{https://youtu.be/IPoh5C9CcI8}.

\bibitem{Bolan2022}
M.~Bolan.
\newblock {Stan Wagon 1321 Solution}.
\newblock {Memo}, September 2022.
\newblock Available electronically at
  \url{https://oeis.org/A352178/a352178.pdf}.

\bibitem{Bremner2011}
Murray~R. Bremner.
\newblock {\em {Lattice Basis Reduction: An Introduction to the LLL Algorithm
  and Its Applications}}.
\newblock CRC Press, 2011.

\bibitem{Clapham1989}
C.~R.~J. Clapham, A.~Flockhart, and J.~Sheehan.
\newblock Graphs without four-cycles.
\newblock {\em {Journal of Graph Theory}}, 13(1):29--47, 1989.
\newblock \href {https://doi.org/10.1002/jgt.3190130107}
  {\path{doi:10.1002/jgt.3190130107}}.

\bibitem{Wagon2021}
{D. Ullman and S. Wagon}.
\newblock {Problem 1321: Powers of Two}.
\newblock {Macalester College Problem of the Week}, April 2021.
\newblock Available electronically at
  \url{https://oeis.org/A347301/a347301_1.pdf}.

\bibitem{Melaih2022}
{F. Melaih}.
\newblock {On The OEIS Sequence A352178}.
\newblock {Memo}, September 2022.
\newblock Available electronically at
  \url{https://oeis.org/A352178/a352178_3.pdf}.

\bibitem{nauty}
B.~D. McKay and A.~Piperno.
\newblock {Practical graph isomorphism, II}.
\newblock {\em {Journal of Symbolic Computation}}, 60:94--112, 2014.
\newblock \href {https://doi.org/10.1016/j.jsc.2013.09.003}
  {\path{doi:10.1016/j.jsc.2013.09.003}}.

\bibitem{web_girth5}
M.~Meringer.
\newblock Connected regular graphs with girth at least 5.
\newblock
  \url{https://www.mathe2.uni-bayreuth.de/markus/reggraphs.html#GIRTH5}.
\newblock Accessed: 2025-01-20.

\bibitem{Meringer1999}
M.~Meringer.
\newblock Fast generation of regular graphs and construction of cages.
\newblock {\em Journal of Graph Theory}, 30(2):137--146, 1999.
\newblock \href
  {https://doi.org/10.1002/(SICI)1097-0118(199902)30:2<137::AID-JGT7>3.0.CO;2-G}
  {\path{doi:10.1002/(SICI)1097-0118(199902)30:2<137::AID-JGT7>3.0.CO;2-G}}.

\bibitem{Reiman1958}
I.~Reiman.
\newblock {Über ein Problem von K. Zarankiewicz}.
\newblock {\em Acta Mathematica Academiae Scientiarum Hungarica},
  9(3):269--273, 1958.
\newblock \href {https://doi.org/10.1007/BF02020254}
  {\path{doi:10.1007/BF02020254}}.

\bibitem{Sage}
{SageMath}.
\newblock {\em {version {\tt 10.7}}}, 2025.
\newblock \url{https://www.sagemath.org/}.

\bibitem{ShahAli2021}
H.~A. ShahAli and S.~Wagon.
\newblock Problem 12272.
\newblock {\em The American Mathematical Monthly}, 128(8):755--763, 2021.
\newblock \href {https://doi.org/10.1080/00029890.2021.1953285}
  {\path{doi:10.1080/00029890.2021.1953285}}.

\bibitem{OEIS}
{The OEIS Foundation}.
\newblock {The On-Line Encyclopedia of Integer Sequences}.
\newblock \url{http://oeis.org}, 2025.

\bibitem{Nuffelen1976}
C.~Van~Nuffelen.
\newblock On the incidence matrix of a graph.
\newblock {\em {IEEE Transactions on Circuits and Systems}}, 23(9):572--572,
  1976.
\newblock \href {https://doi.org/10.1109/TCS.1976.1084251}
  {\path{doi:10.1109/TCS.1976.1084251}}.

\end{thebibliography}

\end{document}